\documentclass[12pt, reqno]{amsart}
\usepackage{amssymb}
\usepackage{epsfig}
\usepackage{graphicx}
\numberwithin{equation}{section}
\usepackage{cite}
\usepackage{fullpage}

\input xy
\xyoption{all}

\theoremstyle{plain}
\newtheorem{prop}{Proposition}[section]

\newtheorem{theo}[prop]{Theorem}
\newtheorem{coro}[prop]{Corollary}

\newtheorem{lemm}[prop]{Lemma}

\theoremstyle{definition}
\newtheorem{defi}[prop]{Definition}

\numberwithin{equation}{section}

\newcommand{\bA}{\mathbb A}

\newcommand{\bP}{\mathbb P}
\newcommand{\bQ}{\mathbb Q}
\newcommand{\bZ}{\mathbb Z}
\newcommand{\bR}{\mathbb R}

\newcommand{\sH}{\mathsf H}

\newcommand{\sZ}{\mathsf Z}

\newcommand{\ra}{\rightarrow}

\newcommand{\SL}{{\rm SL}}
\newcommand{\SO}{{\rm SO}}

\newcommand{\Moduli}{\overline{\mathcal{M}}}
\newcommand{\PGL}{{\rm PGL}}

\begin{document}
\title[$\textnormal{SL}_2$-orbit closures of binary forms]
{The distribution of $S$-integral points on $\textnormal{SL}_2$-orbit closures of binary forms}

\author{Sho Tanimoto}
\address{Department of Mathematics\\
Rice University, MS 136 \\
Houston, Texas  77251-1892 \\
USA}
\email{sho.tanimoto@rice.edu}
\author{James Tanis}
\address{Department of Mathematics\\
Rice University, MS 136 \\
Houston, Texas  77251-1892 \\
USA}
\email{jt13@rice.edu}

\date{\today}

\begin{abstract}
We study the distribution of $S$-integral points on $\SL_2$-orbit closures of binary forms and prove  an asymptotic formula for the number of $S$-integral points. This extends a result of Duke, Rudnick, and Sarnak. The main ingredients of the proof are the method of mixing developed by Eskin-McMullen and Benoist-Oh, Chambert-Loir-Tschinkel's study of asymptotic volume of height balls, and Hassett-Tschinkel's description of log resolutions of $\SL_2$-orbit closures of binary forms.
\end{abstract}

\maketitle


\section*{Introduction}
\label{sect:intro}

The distribution of integral points on homogeneous spaces has been studied by several researchers, and important developments are \cite{DRS93} and \cite{EM93}, which used different techniques to settle the problem of asymptotic formulae for the number of integral points on affine symmetric spaces. \cite{EM93} uses an ergodic theoretic approach based on {\it mixing} and this method is extended to the $S$-integral points setting in \cite{BO12}. On the other hand, another approach based on the height zeta functions method, has been developed in \cite{CT-additive}, \cite{CT-toric}, \cite{TBT11}. The advantage of this method is that one can analyze more general $(D, S)$-integral points while ergodic methods are only available when $D$ is the full boundary divisor. However, the height zeta functions method is also limited in that it is only applicable to bi-equivariant compactifications of connected linear algebraic groups.

In this paper, we study $(D, S)$-integral points on one sided equivariant compactifications of connected semisimple groups assuming that $D$ is the full boundary divisor.
Our method is a variant of the method of mixing in \cite{EM93} and \cite{BO12}.
To demonstrate our method, we solve the problem of counting $S$-integral points of bounded height on $\SL_2$-orbit closures of binary forms, which is considered in the integral case by Duke, Rudnick and Sarnak in \cite{DRS93}. 

Let us explain the problem in detail.
Let $V$ be a two-dimensional vector space over $\bQ$ with coordinates $x$ and $y$.
We consider the standard $\SL_2$ action on $V$.
Let $W_n = \mathrm{Sym}^n(V^*)$ be the space of binary forms of degree $n$,
$$
f = a_0x^n + a_1 x^{n-1}y + \cdots + a_n y^n,
$$
where $n \geq 3$.
Here the left action of $\SL_2$ on $W_n$ is given by
$$
(g \cdot f) \begin{pmatrix}
x \\ y
\end{pmatrix} = 
f \left( g^{-1}\begin{pmatrix}
x \\ y
\end{pmatrix}\right).
$$
Consider the projective space $\bP(W_n) \cong \bP^n$, and let $[f] \in \bP(W_n)(\bQ)$
be a binary form of degree $n$ with coefficients in $\bQ$ and with distinct roots.
We define $X_f$ to be the Zariski closure of the $\SL_2$-orbit of $[f]$, i.e, 
$$
X_f = \overline{\SL_2 \cdot [f]} \subset \bP(W_n).
$$
This is a projective threefold defined over $\bQ$.
The complement of the open orbit $U_f:= \SL_2 \cdot [f]$ in $X_f$ forms a geometrically irreducible divisor, and set theoretically it is an intersection of $X_f$ with the discriminant divisor in $\bP(W_n)$.(See \cite[Lemma 1.5]{MU83}.)
We denote this Weil divisor by $D_f$.
We fix the integral model $\bP_{\bZ} (W_n)$ of $\bP(W_n)$ as
$$
\bP_{\bZ} (W_n) := \mathrm{Proj}(\bZ [a_0, \cdots, a_n]),
$$
and let $\mathcal X_f$ and $\mathcal D_f$ be closures of $X_f$ and $D_f$ in $\bP_{\bZ} (W_n)$ respectively. They form flat integral models of $X_f$ and $D_f$ respectively, and we define 
$$
\mathcal U_f = \mathcal X_f \setminus \mathcal D_f.
$$
Let $S$ be a finite set of places including the archimedean place,
and we denote the ring of $S$-integers by $\bZ_S$.
One can consider the counting function of the number of $S$-integral points with respect to a height function $\mathsf H : \bP(W_n)(\bQ) \ra \bR_{>0}$ :
$$
N(\mathcal U_f, T) = \# \{ P \in \mathcal U_f(\bZ_S) \mid \mathsf H (P) \leq T\},
$$
where the height function $\mathsf H$ is introduced in Section \ref{sec:heights}.
In \cite{DRS93}, Duke, Rudnick, and Sarnak studied the asymptotic formula of this counting function when $S$ consists of the archimedean place:
\begin{theo}{\cite[Theorem 1.9]{DRS93}}
\label{theo:DRS}
When $S =\{ \infty\}$,
there exists a constant $c\geq 0$ such that 
$$
N(\mathcal U_f, T) \sim cT^{\frac{2}{n}}.
$$
\end{theo}
Duke, Rudnick, and Sarnak studied the counting problem of integral points on affine symmetric spaces(\cite[Theorem 1.2]{DRS93}) using techniques from automorphic forms. The above remarkable theorem is an example of an asymptotic formula for a non-symmetric space. Their method is based on equidistribution of lattice points in angular sectors on the hyperbolic plane and elementary approximation arguments using the polar decomposition of $\SL_2(\bR)$. In this paper, we give a new proof of Theorem~\ref{theo:DRS} and extend their result to any $S$:

\noindent
\begin{theo}
Suppose that $n \geq 5$ and $f$ is general so that the stabilizer of $[f]$ is trivial.
Then there exists a constant $c \geq 0$ such that 
$$
N(\mathcal U_f, T) \sim cT^{\frac{2}{n}} (\log T)^{\# S - 1}.
$$
\end{theo}

Our proof is based on the method of mixing developed in \cite{EM93}, \cite{BO12}. In \cite{EM93}, Eskin and McMullen introduced axiomatic treatments of the counting problem and the method of mixing. Using mixing, they independently solved the question of distribution of integral points on affine symmetric spaces(\cite[Theorem 1.4]{EM93}). Benoist and Oh generalized this method to $S$-adic Lie group settings in \cite{BO12}, and solved the counting problem of $S$-integral points on affine symmetric spaces $H\backslash G$(\cite[Theorem 1.1, Corollary 1.2, Theorem 1.4]{BO12}). An important property used in their proofs is the wavefront property for symmetric spaces(\cite[Definition 3.1, Proposition 3.2]{BO12}). It is the key to establishing equidistribution of translations of $H$-orbits. We consider a special height function which is invariant under the action of a compact subgroup $H$ satisfying the wavefront property, and reduce the counting problem on $G$ to the counting problem on $H\backslash G$, where $G$ is a $S$-adic Lie group associated to $\SL_2$. This proves that the function $N(\mathcal U_f, T)$ is approximated by the asymptotic volume of height balls.

The computation of asymptotic volume of height balls is a subject of \cite{volume}.
Chambert-Loir and Tschinkel showed that the global geometric data, which is so-called the Clemens complex, controls the volume of height balls, and one can compute the asymptotic formula based on that. However, 
to use their machinery, we need to describe a log resolution of singularities for a pair $(X_f, D_f)$. This has been discussed in \cite{HT03}. The variety $X_f$ admits a moduli interpretation as a subvariety of a moduli space of stable maps, and Hassett and Tschinkel used this moduli interpretation to construct a log resolution of a pair $(X_f, D_f)$. We will recall their result and provide its refinement in Section~\ref{sec:moduli}. It is straightforward to generalize our method to arbitrary number fields, but we restrict ourselves over the field of rational numbers for notational reasons.

Let us outline the contents of the paper. In Section~\ref{sec:heights}, we define the height function $\mathsf H$ and discuss its basic properties. Then in Section~\ref{sec:methodofmixing} we explain the method of mixing and its application. In Section~\ref{sec:moduli}, we recall the construction of a log resolution of a pair $(X_f, D_f)$ in \cite{HT03} and explain how moduli spaces of stable curves and stable maps can be used to obtain a log resolution of $(X_f, D_f)$. In Section~\ref{sec:height balls}, we recall results of \cite{volume} and apply them to obtain asymptotic formulae. In Section~\ref{sect:generalizations}, we discuss some generalizations of results in Section~\ref{sec:methodofmixing}.

\

{\bf Acknowledgments.} 
The authors would like to thank Natalie Durgin, Brendan Hassett, Brian Lehmann, Yuri Tschinkel,
and Anthony V\'{a}rilly-Alvarado for useful discussions. The authors also would like to thank Shinya Koyama and Nobushige Kurokawa for teaching us the reference \cite{Baker} for Liouville numbers.

\section{Height functions}
\label{sec:heights}

In this section, we introduce a height function of $\mathcal O(1)$ on $\bP(W_n)$ to count $S$-integral points on $\mathcal U_f$. First let us recall some definitions regarding height functions in general.
\begin{defi}{\cite[Section 2.1.3]{volume}}
Let $F$ be a locally compact field and $X$ a smooth projective variety defined over $F$. One can consider $X(F)$ as a compact analytic manifold over $F$. Let $L$ be a line bundle on $X$. The $L(F)$ is endowed with the structure of the analytic line bundle on $X(F)$. A metric on $L(F)$ to be a collection of functions $L_P(F) \rightarrow \bR_+$ for all $P \in X(F)$, denoted by $l \mapsto \| l\|$, such that
\begin{itemize}
\item $\| \cdot \|$ is a norm on the $F$-vector space $ L_P(F)$;
\item for any open subset $U\subset X(F)$ and any non vanishing analytic section $\mathsf f \in \Gamma(U,  L(F))$, the function $U \ni P \mapsto \| \mathsf f(P) \|$ is smooth, i.e., it is locally constant if $F$ is non-archimedean, otherwise it is $C^\infty$.
\end{itemize}
\end{defi}

With metrizations, one can define local height functions:
\begin{defi}{\cite[Section 2.2.6]{volume}}
Let $F$ be a locally compact field and $X$ a smooth projective variety defined over $F$, $\mathcal L = (L, \| \cdot \|)$ a metrized line bundle on $X$, and a nonzero section $\mathsf f \in \Gamma(X, \mathcal L)$. Let $U$ be the complement of the support of $\mathsf f$. The local height function of $\mathcal L$ associated to $\mathsf f$ is given by
$$
\sH : U(F) \rightarrow \bR_+, \quad P \mapsto \| \mathsf f(P) \|^{-1}.
$$
\end{defi}
We define height functions of $\mathcal O(1)$ on $\bP(W_n)$.
For a nonarchimedean place $v$, we define a metrization on $\mathcal O(1)$ by requiring the following property: for any linear form $\mathsf f \in \Gamma(\mathcal O(1), \bP(W_n))$, we have
\[
\|\mathsf f\| (a_0, \cdots, a_n) = \frac{|\mathsf f(a_0, \cdots, a_n)|_v}{\max \{ |a_0|_v, \cdots, |a_n|_v \}}
\]
At the archimedean place, we define our metrization by
\[
\|\mathsf f\| (a_0, \cdots, a_n) = \frac{|\mathsf f(a_0, \cdots, a_n)|_v}{\sqrt{\sum_{i = 0}^n \binom{n}{i}^{-1} a_i^2}}
\]
For $v = p$ a prime, $\mathcal O(1)$ is endowed with the standard metric induced from the integral model $\bP_\bZ(W_n)$.(\cite[Section 2.3]{volume})

Let $D \subset \bP(W_n)$ be the discriminant divisor and $s_D$ the corresponding section of $\mathcal O(D)$. The section $s_D$ is a homogeneous polynomial of degree $2n-2$ with $\bZ$-coefficients. Let $S$ be a finite set of places including the archimedean place and $U = \bP(W_n) \setminus D$. For each $v \in S$, we define the local height $\mathsf H_v : U(\bQ_v) \ra \bR_{>0}$ associated to $\frac{1}{2n-2}D$ by
$$
\mathsf H_v(a_0 , \cdots , a_n) = \frac{ \max \{ |a_0|_v, \cdots, |a_n|_v \} }{ |s_D(a_0, \cdots, a_n)|_v^{\frac{1}{2n-2}}}
$$
when $v$ is a non-archimedean place, and
$$
\mathsf H_v(a_0 , \cdots , a_n) = \frac{ \sqrt{\sum_{i = 0}^n \binom{n}{i}^{-1} a_i^2}}{ |s_D(a_0, \cdots, a_n)|_{\infty}^{\frac{1}{2n-2}}}
$$
when $v$ is the archimedean place, where 
$$
[a_0x^n + a_1 x^{n-1}y + \cdots + a_ny^n] \in U(\bQ_v).
$$
The function $\mathsf H_v$
is the local height function of $\mathcal O(1)$ associated to $\frac{1}{2n-2}D$.

One important property of these local heights is that they are invariant under the action of a maximal compact subgroup:
\begin{lemm}
\label{lemm:invariant}
For $v = p$ a prime, $\mathsf H_p$ is invariant under the action of $\SL_2(\bZ_p)$.
For $v = \infty$, $\mathsf H_\infty$ is invariant under the action of $\mathrm{SO}_2(\bR)$.
\end{lemm}
\begin{proof}
Let $\mathcal L$ be the metrized line bundle associated to the invertible sheaf $\mathcal O(1)$ on $\bP(W_n)$.
First note that $s_D$ is $\SL_2$-invariant, i.e., for any $P \in \bP(W_n)(\bQ_v)$ and $g \in \SL_2(\bQ_v)$, we have
\[
g^{-1} \cdot (s_D( g\cdot P) ) = s_D(P),
\]
where the group $\SL_2$ acts on the line bundle $\mathcal L$ in a way that the action is compatible with the one on $\bP(W_n)$.
Suppose that the place $v = p$ is a non-archimedean place. The group scheme $\SL_2 / \mathrm{Spec}(\bZ_p)$ acts on $\bP_{\bZ_p}(W_n)$ as well as the line bundle $\mathcal L / \mathrm{Spec}(\bZ_p)$, hence $\SL_2(\bZ_p)$ acts on the $\bZ_p$-bundle $\mathcal L(\bZ_p)$.
This means that the action of $\SL_2(\bZ_p)$ on $\mathcal L (\bQ_p)$ is isometric.
Thus using the definition of the height function above, we obtain that for $P\in U(\bQ_p)$ and $g \in \SL_2(\bZ_p)$,
\[
\sH_p(gP) =\|s_D(gP) \|_p^{-1} = \|g^{-1}\cdot s_D(gP) \|_p^{-1} = \|s_D(P) \|_p^{-1} = \sH_p(P).
\]
For the archimedean place, see \cite[Section 4]{DRS93} where this special height function was used to reduce the counting problem on $\SL_2$ to the counting problem on the hyperbolic plane $\mathfrak H$.
\end{proof}

Define $\bA_S := \prod_{v\in S} \bQ_v$, and we call it the $S$-adic ring.
Let $U(\bA_S) = \prod_{v \in S} U(\bQ_v)$ be a $S$-adic manifold,
and we define the global height $\mathsf H : U(\bA_S) \ra \bR_{>0}$ as the product of local heights :
$$
\mathsf H(P) := \prod_{v \in S} \mathsf H_v(P_v),
$$
where $P = (P_v)_{v\in S} \in U(\bA_S)$.
This is a continuous function on $U(\bA_S)$.
A key property of height functions 
is that the set of $S$-integral points of bounded height is finite,
hence we can count the number of $S$-integral points of bounded height:
$$
N(\mathcal U_f, T) = \# \{ P \in \mathcal U_f(\bZ_S) \mid \mathsf H(P) \leq T \},
$$
for any binary form $f$ of degree $n \geq 3$ with distinct roots.
We are interested in studying the asymptotic behavior of $N(\mathcal U_f, T)$ as $T \ra \infty$.

\section{The method of mixing}
\label{sec:methodofmixing}
Let $\mathcal D$ be the integral model of the discriminant divisor in $\bP_\bZ(W_n)$
and $\mathcal U$ the complement of $\mathcal D$.
By a theorem of Borel, Harish-Chandra, there are only finitely many $\SL_2 (\bZ_S)$-orbits on $\mathcal U_f (\bZ_S)$.(See \cite[Theorem 6.9]{BHC62} and \cite[Theorem 5.8]{PR94}.)
Also, the stabilizer of $f$ is a finite group. Therefore, our problem concerning the asymptotic of $N(\mathcal U_f, T)$ can be reduced to evaluating the asymptotic of 
$$
N(f, T) = \#\{ g \in \SL_2(\bZ_S) \mid \mathsf H(g \cdot [f]) \leq T \},
$$
for any $f \in \mathcal U(\bZ_S)$.

To study this counting function, we will use the method of mixing
developed in \cite{EM93} and \cite{BO12} for symmetric varieties.
Specifically, \cite{EM93} used mixing and the so-called wavefront property to study 
the distribution of integral points for sufficiently "nice" sets, 
and \cite{BO12} developed this theory for $S$-integral points.

Let $G = \SL_2(\bA_S) = \prod_{v \in S} \SL_2(\bQ_v)$ be a $S$-adic Lie group
and $\Gamma = \SL_2(\bZ_S)$ be diagonally embedded in $G$.
Then $\Gamma$ is a lattice in $G$, i.e., $\Gamma$ is discrete in $G$ and $X := \Gamma\backslash G$ has finite volume with respect to the invariant measure $\mu_X$ on $X$.(See \cite[Theorem I.3.2.4]{Mar91}.)
For the group action of $G$, the mixing property is as follows:
\begin{theo}
\label{theo:mixing}
The action of $G$ on $X$ is mixing, i.e., for any $\alpha, \beta \in \mathsf L^2(X)$, we have
$$
\lim_{g \ra \infty} \int_X \alpha(xg) \beta(x) \, \mathrm d \mu_X(x)
= \frac{\int_X \alpha \, \mathrm d \mu_X \int_X \beta\, \mathrm d \mu_X}{\mu_X(X)}.
$$
\end{theo}
\begin{proof}
This is a consequence of \cite[Proposition 2.4]{BO12}.
To apply \cite[Proposition 2.4]{BO12}, one needs to check two conditions: (i) our $S$-adic Lie group $G$ satisfies the Howe-Moore property; (ii) our lattice $\Gamma$ is irreducible. 

The condition (i) is stated in \cite[Theorem 2.5]{BO12} which claims the Howe-Moore property for any $S$-adic Lie group associated to a semisimple group. 

To check the condition (ii), see \cite[Lemma 9.4]{BO12} whose assumptions are all satisfied by $\SL_2$.
\end{proof}

To count $S$-integral points, we consider the height balls:
$$
\mathsf B(T) = \{ g \in G \mid \mathsf H(g \cdot [f]) \leq T \}.
$$
We denote the volume of these height balls with respect to the Haar measure $\mu_G$ on $G$ by $V(T)$. Here we assume that $\mu_G = \mu_X$ holds locally. The asymptotic of this volume function will be studied in Section~\ref{sec:height balls}.
To apply results of \cite{BO12}, these height balls need to satisfy the following condition:
\begin{prop}
\label{prop:wellrounded}
For any $\epsilon > 0$, there exists a neighborhood $U$ of the identity $e$ in $G$ such that
$$
(1 - \epsilon) \mu_G (\cup_{g \in U} \mathsf B(T)g ) \leq V(T) \leq (1+\epsilon) \mu_G (\cap_{g \in U} \mathsf B(T)g),
$$
for all $T \gg 1$.
\end{prop}
\begin{proof}
The above condition is refereed as {\it well-roundedness} in \cite{EM93} and \cite{BO12}. This follows from the precise asymptotic formula in Corollary~\ref{coro:precise}. Indeed for any $\epsilon >0$, there exists $\delta >0$ such that 
$$
\frac{V(T+\delta)}{V(T)} < \frac{1}{1-\epsilon}, \quad \frac{V(T)}{V(T-\delta)} < 1+ \epsilon
$$
for sufficiently large $T \gg 1$. Now choose a neighborhood $U$ of $e$ so that
\[
\cup_{g \in U} \mathsf B(T)g \subset \mathsf B(T+\delta), \quad \mathsf B(T-\delta) \subset \cap_{g \in U} \mathsf B(T)g.
\]
Our assertion follows from this.
\end{proof}

Let $\chi_{\mathsf B(T)}$ be the characteristic function of the height ball $\mathsf B(T)$.
Consider the counting function:
$$
F_T(g) = \sum_{\gamma \in \Gamma} \chi_{\mathsf B(T)}(\gamma g) = \# \Gamma \cap (\mathsf B(T)g^{-1}).
$$
This defines a function on $X = \Gamma \backslash G$.
Our goal in this section is to prove the following theorem:
\begin{theo}
\label{theo:main_mixing}
Let $V^*(T) = V(T)/\mu_X(X)$.
Then we have point-wise convergence
$$
\frac{F_T(g_0)}{V^*(T)} \ra 1 \textnormal{\quad as $T\ra \infty$},
$$
for any $g_0 \in G$.
\end{theo}

To prove this theorem, we consider a compact subgroup of $G$ satisfying the wavefront property as in \cite[Definition 3.1]{BO12}:
\begin{defi}
\label{defi:wavefront}
Let $\mathcal G$ be a locally compact group, and $\mathcal H$ a closed subgroup of $\mathcal G$. The group $\mathcal G$ has the wavefront property in $\mathcal H\backslash \mathcal G$ if there exists a Borel subset $F$ such that $\mathcal G = \mathcal HF$ and, for every neighborhood $U$ of the identity in $\mathcal G$, there exists a neighborhood $V$ of the identity such that
$$
\mathcal H V g \subset \mathcal HgU,
$$
for all $g \in F$.
\end{defi}

This property first appeared in \cite{EM93} to establish equidistribution of $\mathcal H$-orbits:

\begin{theo}{\cite[Theorem 4.1]{BO12}}
\label{theo:equidistribution}
Let $\mathcal G$ be a locally compact group, $\mathcal H \subset \mathcal G$ a closed subgroup, $\Gamma \subset \mathcal G$ a lattice such that $\Gamma_{\mathcal H} = \Gamma \cap \mathcal H$ is a lattice in $\mathcal H$. Set $\mathcal X := \Gamma \backslash \mathcal G$ and $\mathcal Y := \Gamma_{\mathcal H} \backslash \mathcal H$, and we denote the invariant measures on $\mathcal X$ and $\mathcal Y$ by $\mu_{\mathcal X}$ and $\mu_{\mathcal Y}$ respectively.

Suppose that the action of $\mathcal G$ on $\mathcal X$ is mixing and that $\mathcal G$ has the wavefront property on $\mathcal H \backslash \mathcal G$. Then the translates $\mathcal Yg$ become equidistributed in $\mathcal X$ as $g \rightarrow \infty$ in $\mathcal H \backslash \mathcal G$, i.e., for any $\psi \in C_c(\mathcal X)$, we have
\[
\frac{1}{\mu_{\mathcal Y}(\mathcal Y)} \int_{\mathcal Y} \psi (yg) \, \mathrm d \mu_{\mathcal Y}(y) \rightarrow \frac{1}{\mu_{\mathcal X}(\mathcal X)} \int_{\mathcal X}\psi \, \mathrm d \mu_{\mathcal X},
\]
as the image of $g$ in $\mathcal H \backslash \mathcal G$ leaves every compact subset.
\end{theo}
Define a subgroup $H$ of $G$ by
$$
H = \prod_{p \in S_{\mathrm{fin}}} \SL_2(\bZ_p) \times \SO_2 (\bR),
$$
where $S_{\mathrm{fin}}= S \setminus \{\infty\}$.
Note that Lemma~\ref{lemm:invariant} implies that our height function $\mathsf H$ is invariant under the action of the subgroup $H$, i.e.,
$$
\mathsf H( hg \cdot [f] ) = \mathsf H(g \cdot [f])
$$
for $g \in G$ and $h \in H$.
Moreover, we have
\begin{lemm}
\label{lemm:wavefront}
The group $G$ has the wavefront property in $H\backslash G$.
\end{lemm}
\begin{proof}
For $\SO(2)$, this property is established in \cite[Theorem 3.1]{EM93}.
It follows from the definition of the wavefront property that $\SL_2(\bQ_p)$ has the wavefront property in $\SL_2(\bZ_p)\backslash \SL_2(\bQ_p)$ because $\SL_2(\bZ_p)$ is open. Now our assertion follows from \cite[Proposition 3.5]{BO12} which claims that if each factor satisfies the wavefront property, then their product also satisfies the wavefront property.
\end{proof}
We consider any sequence of non-negative integrable functions $\varphi_n$ on $G$ satisfying
\begin{itemize}
\item
each $\varphi_n$ has a compact support,
\item
$\max_n \| \varphi_n \|_\infty < \infty$,
\item
$\lim_{n \ra \infty} \int_G \varphi_n \, \mathrm d \mu_G = \infty$,
\item
$\varphi_n(hg) = \varphi_n(g)$ for any $g \in G$ and $h \in H$.
\end{itemize}
The sequence $\{\chi_{\mathsf B(T_n)}\}$ is an example.
We define a function $F_{\varphi_n}$ on $X = \Gamma \backslash G$ by
$$
F_{\varphi_n}(x) := \sum_{\gamma \in \Gamma}\varphi_n(\gamma g),
$$
where $x = \Gamma g$.
Let $I_n := \int_G \varphi_n \, \mathrm d \mu_G/\mu_X(X)$.
The following proposition is essentially proved in \cite[Proposition 5.3]{BO12}.
The only difference is that we are not counting $S$-integral points on $H \backslash G$, but on $G$.
\begin{prop}
\label{prop:weakconvergence}
$F_{\varphi_n}/I_n$ converges weakly to 1 as $n \ra \infty$.
\end{prop}

\begin{proof}
Let $\alpha \in C_c(X)$. We want to prove that
$$
\lim_{n \ra\infty} \frac{1}{I_n} \int_X F_{\varphi_n}(x) \alpha (x) \, \mathrm d \mu_X(x) = \int_X \alpha(x) \, \mathrm d \mu_X(x).
$$
Let $\mu_H$ and $\mu_{H \backslash G}$ be the Haar measure and the invariant measure on $H$ and $H \backslash G$ respectively such that $\mu_G = \mu_H \mu_{H \backslash G}$ and $\mu_H(H) = 1$.
\begin{align*}
\int_X F_{\varphi_n} \alpha \, \mathrm d \mu_X &= \int_{\Gamma \backslash G} \sum_{\gamma \in \Gamma} \varphi_n (\gamma g) \alpha (\Gamma g) \, \mathrm d \mu_X(\Gamma g)\\
&= \int_G \varphi_n(g) \alpha(g) \, \mathrm d \mu_G(g)\\
&= \int_{H \backslash G} \varphi_n(Hg) \int_H \alpha (hg) \, \mathrm d \mu_H(h)\, \mathrm d \mu_{H \backslash G}(Hg)\\
&= \# \Gamma_H \int_{H \backslash G} \varphi_n(Hg) \int_{\Gamma_H \backslash H} \alpha (\Gamma_H hg) \, \mathrm d \mu_H(h)\, \mathrm d \mu_{H \backslash G}(Hg),
\end{align*}
where $\Gamma_H = \Gamma \cap H$. Note that since $H$ is compact, $\Gamma_H$ is a finite group and a lattice in $H$. Since the action of $G$ on $X$ is mixing and $G$ has the wavefront property in $H \backslash G$, the translates $\Gamma_H \backslash H g$ become equidistributed in $X$ as $g \ra \infty$ in $H \backslash G$(Theorem~\ref{theo:equidistribution}).
This means that
$$
\int_{\Gamma_H \backslash H} \alpha (\Gamma_H hg) \, \mathrm d \mu_H(h) \ra \frac{\mu_H(H)}{\#\Gamma_H \mu_X(X) } \int_X \alpha \, \mathrm d \mu_X,
$$
as $g \ra \infty$.
Since $\varphi_n$ is uniformly bounded, our assertion follows from the dominated convergence theorem.
\end{proof}
Finally we prove Theorem \ref{theo:main_mixing}:
\begin{proof}
The following proof mirrors the proof of \cite[Proposition 6.2]{BO12}.
Fix $\epsilon > 0$. Let $U$ be a symmetric neighborhood of the identity $e$ such that
$$
(1 - \epsilon) \mu_G (\cup_{g \in U} \mathsf B(T)g ) \leq V(T) \leq (1+\epsilon) \mu_G (\cap_{g \in U} \mathsf B(T)g).
$$
The existence of such $U$ is guaranteed by Proposition~\ref{prop:wellrounded}.
We consider the functions $\varphi_T^{\pm}$ on $G$ given by
$$
\varphi_T^{+}(g) = \sup_{u \in U} \chi_{\mathsf B(T)} (gu^{-1}), \quad \varphi_T^{-}(g) = \inf_{u \in U} \chi_{\mathsf B(T)} (gu^{-1})
$$
We let $I_T^{\pm} = \int_G \varphi_T^{\pm} \, \mathrm d \mu_G / \mu_X(X)$.
Then we have
$$
(1-\epsilon) I_T^+ \leq V^*(T) \leq (1+\epsilon)I_T^-.
$$
On the other hand, we define
$$
F_T^{\pm}(g) = \sum_{\gamma \in \Gamma} \varphi_T^{\pm}(\gamma g).
$$
It is easy to verify that
$$
F_T^-(gu) \leq F_T(g) \leq F_T^+(gu),
$$
for all $g \in G$ and $u \in U$.
Pick a non-negative continuous function $\alpha$ on $X$ such that $\int_G \alpha \, \mathrm d \mu_G = 1$ and the support of $\alpha$ is included in $g_0U$.
Then we have
$$
\int_X \alpha F_T^- \, \mathrm d \mu_X \leq F_T(g_0) \leq \int_X \alpha F_T^+ \, \mathrm d \mu_X.
$$
Applying Proposition~\ref{prop:weakconvergence} to $\varphi_T^{\pm}$ we obtain that for $T \gg 1$
$$
(1-\epsilon)I_T^- \leq F_T(g_0) \leq (1+\epsilon)I_T^+.
$$
We can conclude that for $T \gg 1$
$$
\frac{1-\epsilon}{1+\epsilon} \leq \frac{F_T(g_0)}{V^*(T)} \leq \frac{1+\epsilon}{1-\epsilon}.
$$
Thus $F_T(g_0)/V^*(T)$ converges to 1.
\end{proof}

\section{Moduli interpretations}
\label{sec:moduli}

In \cite{HT03}, using moduli interpretations, Hassett and Tschinkel constructed a partial desingularization of a pair $(X_f, D_f)$ to give a geometric explanation for the result of Duke, Rudnick and Sarnak (Theorem~\ref{theo:DRS}). This is not exactly what we need, however their geometry is quite important for establishing the asymptotic formula for $V(T)$, the volume of height balls. In this section, we recall the geometry of $(X_f, D_f)$ described in \cite{HT03}, and then provide its refinement. We assume that the ground field is an algebraically closed field of characteristic zero throughout this section.

First let us recall some definitions of stable curves and stable maps. In this paper, a curve of genus zero is a connected projective curve $C$ such that (i) $C$ is the union of a finite number of $\bP^1$'s; (ii) each component of $C$ meets with other components transversally; (iii) $C$ is a tree of smooth rational curves, i.e., there are no loops of rational curves.  

A curve of genus zero with $n$ marked points is a curve $C$ of genus zero together with mutually distinct $n$ smooth points on $C$. A special point on a curve $C$ of genus zero with $n$ marked points $(p_1, \cdots, p_n)$ is either a marked point $p_i$ or an intersection of two irreducible components of $C$. A stable curve of genus zero with $n$ marked points is a curve of genus zero together with $n$ marked points
\[
(C, p_1, \cdots, p_n)
\]
such that each component has at least three special points. This condition is equivalent to say that the automorphism group of $(C, p_1, \cdots, p_n)$ is finite.

\begin{center}
\includegraphics[scale=0.12]{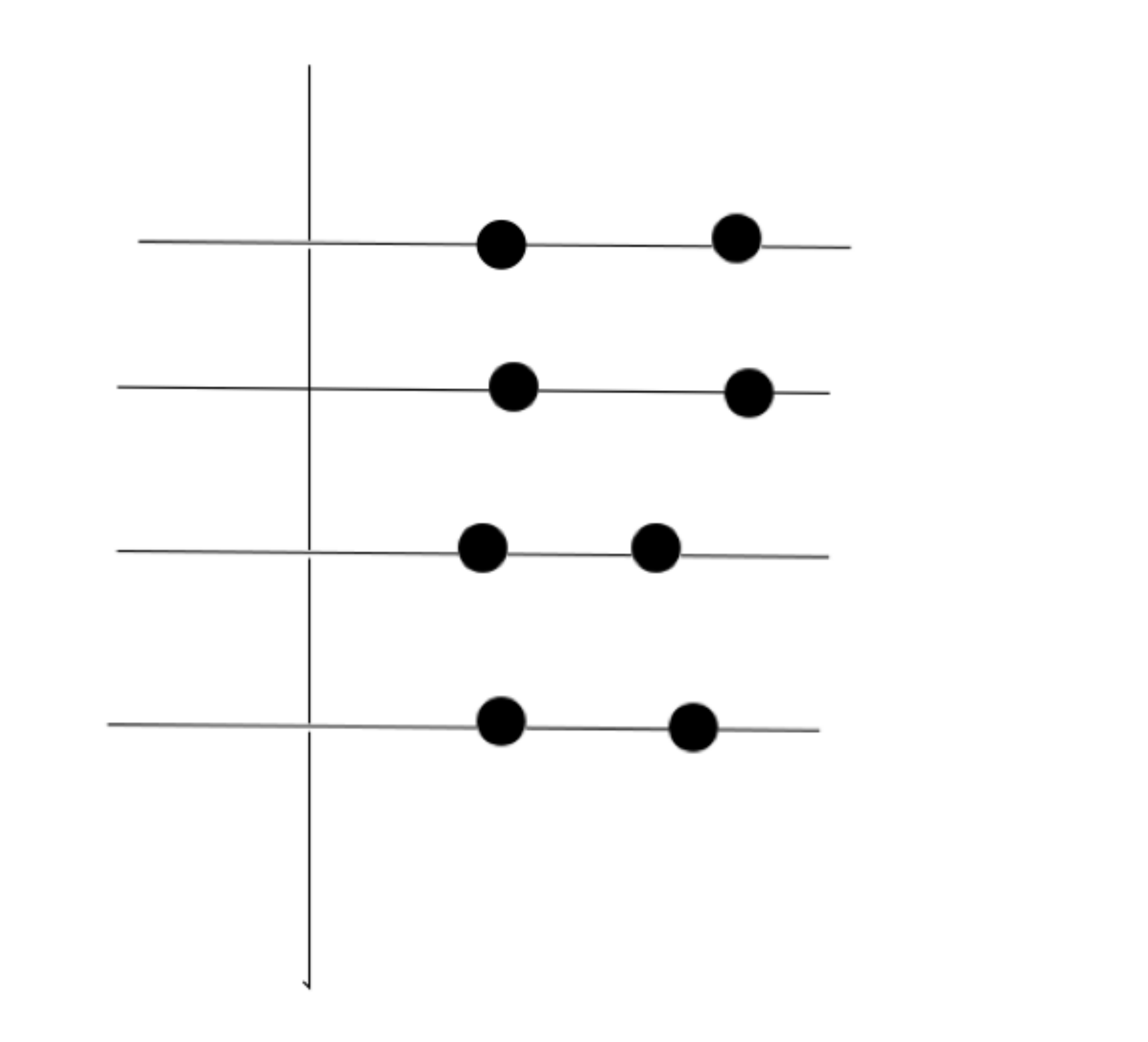}
\end{center}

A stable map of degree one from genus zero curves with $n$ marked points to $\bP^1$ is a tuple
\[
(C, p_1, \cdots, p_n, \mu : C \ra \bP^1).
\]
such that (1) $(C, p_1, \cdots, p_n)$ is a curve of genus zero with $n$ marked points; (ii) the morphism $\mu$ has degree one, i.e., all components except one component $L$ are collapsed by $\mu$ and the restriction of $\mu$ to $L$ is an isomorphism to $\bP^1$; (iii) all components except $L$ has at least three special points. Again this condition (iii) is equivalent to say that the automorphism group of $(C, p_1, \cdots, p_n, \mu)$ is finite.

\begin{center}
\includegraphics[scale=0.6]{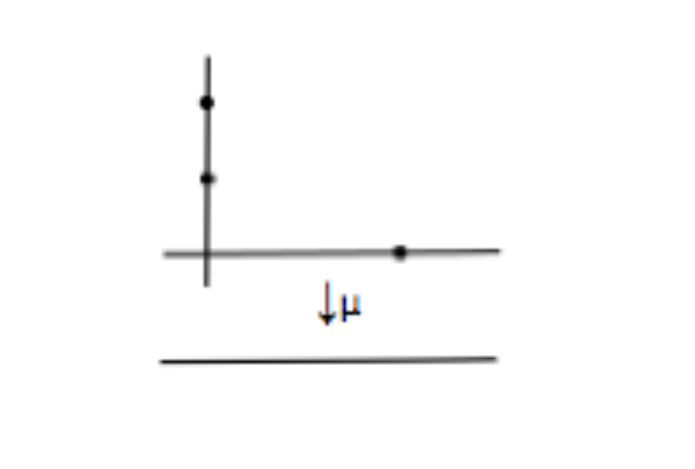}
\end{center}

Fix $n \geq 3$. Let $\Moduli_{0,n}$ be the Knudsen-Mumford moduli space of stable curves of genus zero with $n$ marked points. Let $\Moduli_{0,n}(\bP^1,1)$ denote the Kontsevich moduli space of stable maps of degree one from genus zero curves with $n$ marked points to $\bP^1$.
The action of $\SL_2$ on $\Moduli_{0,n}(\bP^1, 1)$ is defined by
$$
g \cdot (C, p_1, \cdots, p_n, \mu) = (C, p_1, \cdots, p_n, g\circ \mu).
$$
We consider the forgetting map
\begin{align*}
\psi : \Moduli_{0,n}(\bP^1, 1) & \ra \Moduli_{0,n} \\
(C, p_1, \cdots, p_n, \mu) & \mapsto (C', p_1, \cdots, p_n)
\end{align*}
where $C'$ is formed from $C$ by collapsing the irreducible components that are destabilized, i.e., irreducible components with at most two special points.

For each subset $S \subset N = \{1, \cdots, n\}$ such that $|S| \geq 2$, consider stable maps of degree one $\mu :C \ra \bP^1$ such that
\begin{itemize}
\item $C$ is the union of two $\bP^1$s,
\item marked points in $S$ are on the one component, and remaining marked points are on another component,
\item The component containing marked points in $S$ is collapsed by $\mu$.
\end{itemize}

\begin{center}
\includegraphics[scale=0.5]{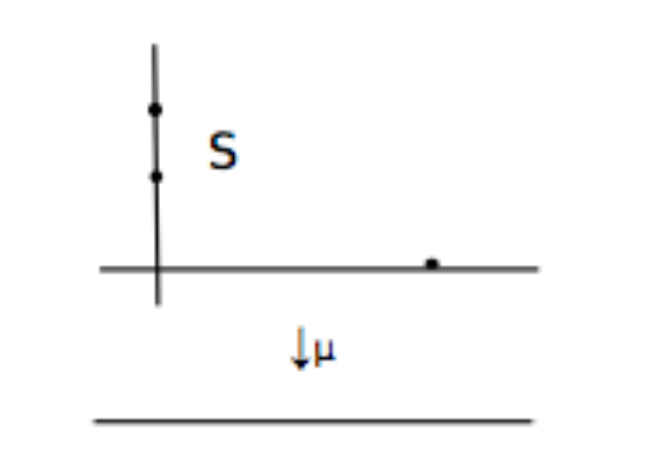}
\end{center}

The Zariski closure of the locus of such stable maps becomes an irreducible divisor $B_S\subset \Moduli_{0,n}(\bP^1,1)$. For $s = 2, \cdots n$, we define
$$
B[s] := \sum_{|S| = s} B_S, \quad B := \sum_{s = 2}^nB[s].
$$
\begin{theo}
\label{theo:moduli_smooth}
The moduli spaces $\Moduli_{0,n}$ and $\Moduli_{0,n}(\bP^1, 1)$ are smooth projective algebraic varieties, and the boundary divisor $B$ is a divisor with strict normal crossings.
\end{theo}
The evaluation map $\mathrm{ev}$

\begin{align*}
\mathrm{ev} :\Moduli_{0,n}(\bP^1, 1) &\ra (\bP^1)^n\\
(C, p_1, \cdots, p_n, \mu) &\mapsto (\mu(p_1), \cdots, \mu(p_n))
\end{align*}
is an $\SL_2$-equivariant birational morphism. This is because for any $(p_1, \cdots, p_n) \in (\bP^1)^n$ such that $p_i$'s are mutually distinct, its preimage by $\mathrm{ev}$ is simply
\[
(\bP^1, p_1, \cdots, p_n, \mathrm{id})
\]
so $\mathrm{ev}$ is an isomorphism on this locus.
The divisor $B[2]$ is the proper transform of the big diagonal $\Delta$ and other boundary divisors $B[s]$($s \geq 3$) are contracted by the evaluation map $\mathrm{ev}$. The symmetric group $\mathfrak S_n$ acts on $\Moduli_{0,n}(\bP^1, 1)$ by permuting marked points. Note that the actions of $\SL_2$ and $\mathfrak S_n$ commute. We consider $\mathfrak S_n$-quotients of $\mathrm{ev}$ to obtain an $\SL_2$-equivariant birational map:
$$
\varrho : \tilde{X} = \Moduli_{0,n}(\bP^1, 1)/\mathfrak S_n \ra (\bP^1)^n/\mathfrak S_n \cong \bP(W_n)
$$
Let $\tilde{D}[s]$ be the image of $B[s]$. Then $\tilde{D}[2]$ is the proper transform of the discriminant divisor $D \subset \bP(W_n)$ and other boundary divisors $\tilde{D}[s]$($s\geq 3$) are contracted by $\varrho$. We write $\tilde{D}$ for $\sum_{s=2}^n \tilde{D}[s]$.

Let $f$ be a binary form of degree $n$ with distinct roots. As we defined in introduction, $X_f$ is the Zariski closure of the $\SL_2$-orbit of $f$ in $\bP(W_n)$. Let $\tilde{X}_f \subset \tilde{X}$ denote the strict transform of $X_f$, and define $\tilde{D}_f = \tilde{X}_f \cap \tilde{D}$. Here the intersection is the set theoretic intersection. A pair $(\tilde{X}_f, \tilde{D}_f)$ is a partial desingularization of $(X_f, D_f)$ constructed by Hassett and Tschinkel, and in particular $(\tilde{X}_f, \tilde{D}_f)$ is log canonical for a generic $f$(\cite[Proposition 2.8]{HT03}). Write $\alpha = (\alpha_1, \cdots, \alpha_n)$ for distinct roots of $f$, and let $C_\alpha \in \Moduli_{0,n}$ be the corresponding pointed rational curve
$$
C_\alpha = (\bP^1, \alpha_1, \cdots, \alpha_n).
$$
We define $Y_\alpha$ to be $\psi^{-1}(C_\alpha)$. Then $\tilde{X}_f$ is a finite quotient of $Y_\alpha$, and $Y_\alpha$ is an equivariant compactification of $\PGL_2$. More precisely, consider the stabilizer of $C_\alpha$ in the symmetric group $\mathfrak S_n$:
$$
\mathfrak H = \{ \sigma \in \mathfrak S_n \mid \sigma \cdot C_\alpha \cong C_\alpha \}.
$$
Then $\tilde{X}_f$ is isomorphic to the quotient $Y_\alpha / \mathfrak H$. When $n \geq 5$ and $f$ is general, the stabilizer $\mathfrak H$ is trivial. Hence for such $f$, we have $\tilde{X}_f \cong Y_\alpha$. Each $Y_\alpha$ only meets with $B[n]$ and $B[n-1]$. Indeed, for any stable map $C=(C, p_1, \cdots, p_n, \mu) \in B[s]$ such that $s \leq n-2$, the image of the forgetting map $\psi(C)$ has at least two irreducible components, and hence it cannot be equal to $C_\alpha$. The following theorem will be used later:
\begin{theo}{\cite[Proposition 2.6]{HT03}}
\label{theo:strictnormalcrossings}
$Y_\alpha$ is smooth and its boundary has strict normal crossings contained in $B[n-1]\cup B[n]$.
\end{theo}
The proof of the above theorem gives us an explicit description of the $Y_\alpha$. When $n = 3$, then $Y_\alpha = \Moduli_{0,3}(\bP^1, 1)$. The evaluation map $\mathrm{ev} : \Moduli_{0,3}(\bP^1, 1) \ra (\bP^1)^3$ is the blow-up of $(\bP^1)^3$ along the small diagonal $\Delta_{\mathrm{small}}$, and $E:=B[3]$ is the exceptional divisor. Since the normal bundle $\mathcal N_{\Delta_{\mathrm{small}}}$ is isomorphic to $\mathcal O(2) \oplus \mathcal O(2)$, $E$ is isomorphic to $\bP^1 \times \bP^1$. Let
$$
\pi_1 : E \ra \bP^1,
$$
be the projection to the cross-ratio of marked points and the node, and let
$$
\pi_2 : E\ra \bP^1,
$$
be the projection to the image of marked points.

For arbitrary $n$, write $E$ for the intersection of $Y_\alpha$ and $B[n]$. This is the Zariski closure of the locus of stable maps $(C, p_1, \cdots, p_n, \mu)$ such that
\begin{itemize}
\item $C$ is the union of two $\bP^1$s,
\item every marked point is on the collapsed component,
\item $\psi(C, p_1, \cdots, p_n, \mu) \cong C_\alpha$.
\end{itemize}
We define $F_i$ for $i = 1, \cdots, n$ as the intersection of $Y_\alpha$ and $B_{N\setminus \{i\}}$. Then $F_i$ is the Zariski closure of the locus of stable maps $(C, p_1, \cdots, p_n, \mu)$ such that
\begin{itemize}
\item $C$ is the union of two $\bP^1$s,
\item every marked point except $i$-th marked point is on the collapsed component, and $i$-th marked point is on the component which is isomorphically mapped to $\bP^1$ by $\mu$,
\item $\psi(C, p_1, \cdots, p_n, \mu) \cong C_\alpha$.
\end{itemize}
Consider the sequence of the forgetting maps:
$$
Y_{\alpha_1, \cdots, \alpha_n} \ra Y_{\alpha_1, \cdots, \alpha_{n-1}} \ra \cdots \ra Y_{\alpha_1, \alpha_2, \alpha_3, \alpha_4} \ra Y_{\alpha_1, \alpha_2, \alpha_3} \cong \Moduli_{0,3}(\bP^1, 1).
$$
Each morphism is a $\SL_2$-equivariant birational morphism, and $\phi_i : Y_{\alpha_1, \cdots, \alpha_i} \ra Y_{\alpha_1, \cdots, \alpha_{i-1}}$ contracts a divisor $F_i$ onto $\bP^1$. Indeed, for $\mathrm{PGL}_2$-equivariant compactifications, the number of boundary components is equal to Picard rank of the underlying variety(\cite[Proposition 5.1]{HTT14}), so each birational morphism $\phi_i$ must be an extremal contraction. Moreover $F_i$ is contracted by $\phi_i$ because if we forget $\alpha_i$, then we cannot recover the value $\mu(\alpha_i)$. Hence $\phi_i$ is a divisorial contraction contracting $F_i$. On $Y_{\alpha_1, \alpha_2, \alpha_3}$, $E$ is identified with $B[3]$ and $F_1+F_2+F_3$ is identified with $B[2]$. Then each $\phi_i$ can be considered as the blow-up of the proper transform of $\pi_1^{-1}(\alpha_i)\subset B[3]$.

This blow-up description shows that $E$, $F_1, \cdots, F_n$ form a basis for $\mathrm{Pic}(Y_\alpha)_\bQ$, and the canonical bundle is equal to
$$
K_{Y_\alpha} = - \sum_{i = 1}^n F_i -2E.
$$
Consider the morphism
$$
\beta : Y_\alpha \subset \Moduli_{0,n}(\bP^1, 1) \ra \tilde{X} \ra \bP(W_n),
$$
and let $L := \beta^* \mathcal O_{\bP(W_n)}(1)$. We have
\begin{prop}{\cite[Lemma 3.3]{HT03}}
\label{prop:coefficients}
The pullback of the hyperplane class is of the form
$$
L = \frac{n-2}{2} \sum_{i = 1}^n F_i + \frac{n}{2}E.
$$
\end{prop}
\begin{proof}
We include a proof for completeness.
Since $L$ is the pullback from a $\mathfrak S_n$-quotient of $\Moduli_{0,n}(\bP^1,1)$, $L$ must take the form
$$
L = a \sum_{i=1}^n F_i + bE.
$$
Note that the scheme-theoretic intersection $Y_\alpha \cap B_{N\setminus\{i\}}$ is reduced. Indeed, the divisor $B_{N\setminus\{i\}}$ is isomorphic to $\Moduli_{0,n} \times \Moduli_{0,2}(\bP^1,1)$, so the intersection $Y_\alpha \cap B_{N\setminus\{i\}}$ is isomorphic to $\Moduli_{0,2}(\bP^1,1) \cong \bP^1 \times \bP^1$. Let $R\subset E$ be the proper transform of the general fiber of $\pi_2 : E \ra \bP^1$. Then we have
$$
F_i . R = 1, \quad \text{and} \quad (E+F_4 + \cdots + F_n).R = -1.
$$
The second identity follows from \cite[Theorem 8.24(c)]{HarAG}. Hence we find $E.R = 2-n$. On the other hand, $R$ is contracted by $\beta$, so we have $L.R = 0$. Thus we conclude that $L$ is of the form
$$
L = c((n-2) \sum_{i=1}^n F_i + nE)
$$
Next, let $C\subset E$ be the proper transform of the general fiber of $\pi_1 : E \ra \bP^1$. The $\beta$ maps $C$ isomorphically onto its image and the image has degree $n$. Thus we obtain $L.C = n$. It is easy to see that $F_i.C = 0$. Moreover, since $\mathcal N_{\Delta_{\mathrm{small}}} \cong \mathcal O(2) \oplus \mathcal O(2)$, it follows from \cite[Proposition 7.12 and Theorem 8.24(c)]{HarAG} that 
$$
\mathcal O(E)|_C \cong \mathcal O(2).
$$
Therefore, we find $E.C =2$ and we conclude that $c = \frac{1}{2}$.
\end{proof}

Let $Z$ be any smooth equivariant compactification of $\SL_2$. The degree 2 map $\SL_2 \ra \rm{PGL}_2$ extends to a $\SL_2$-equivariant rational map $\varphi : Z \dashrightarrow Y_\alpha$. After applying $\SL_2$-equivariant resolution, if necessary, we may assume that $\varphi$ is an honest $\SL_2$-equivariant morphism and the boundary divisor of $Z$ is a divisor with strict normal crossings.

\begin{lemm}
\label{lemm:ramified}
$\varphi : Z \ra Y_\alpha$ is ramified along $E$ and $F_i$ for each $i$.
\end{lemm}
\begin{proof}
We prove that $E$ is ramified. Recall that the forgetting map $Y_\alpha \ra Y_{\alpha_1, \alpha_2, \alpha_3} = \Moduli_{0,3}(\bP^1,1)$ is a $\SL_2$-equivariant birational morphism and $E$ is identified with $B[3]$. Moreover, the evaluation map $\mathrm{ev} : \Moduli_{0.3}(\bP^1,1) \ra (\bP^1)^3$ is the blow-up along the small diagonal $\Delta_{\rm{small}}$ and $E$ is the exceptional divisor of this blow up. Thus, we only need to consider $Z \ra \Moduli_{0,3}(\bP^1,1)$ and prove that the exceptional divisor $B[3]$ is ramified. After changing a base point, if necessary, we may assume that the map $\SL_2 \ra (\bP^1)^3$ is given by
$$
\SL_2 \ni g \mapsto \left(g\begin{pmatrix} 1\\0\end{pmatrix}, g\begin{pmatrix} 0\\1\end{pmatrix}, g\begin{pmatrix} 1\\1\end{pmatrix}\right) \in (\bP^1)^3.
$$
The function field of $Z$ is $k(a,b,c,d)$ where $k$ is the ground field and $a,b,c,d$ satisfy $ad-bc = 1$. Then the function field of $\rm{PGL}_2$ is a subfield of $k(a,b,c,d)$:
$$
k\left(\frac{a}{c}, \frac{b}{d}, \frac{a+b}{c+d} \right).
$$
Let $x = a/c$, $y = b/d$, and $z = (a+b)/(c+d)$. The divisor $E$ is the exceptional locus over the small diagonal $\Delta_{\rm{small}}$ defined by $x-y=0$ and $y-z=0$. Write $t$ for $(y-z)/(x-y)$. Then the local ring $\mathcal O_E$ at $E$ is
$$
k[t,x,y]_{(x-y)}.
$$
Note that $k(a,b,c,d) = k(x,y,z)(c^{-1})$ and $c^{-1}$ satisfies
$$
c^{-2} = -(x-y)(1+t^{-1}).
$$
Therefore, we conclude that $E$ is ramified. The divisors $F_i$ can be studied similarly.
\end{proof}

Let $Z \ra  \tilde{Z} \ra Y_\alpha$ be the stein factorization of $\varphi : Z \ra Y_\alpha$. We denote a $\SL_2$-equivariant birational map $Z \ra \tilde{Z}$ by $h$ and a degree 2 finite morphism $\tilde{Z} \ra Y_\alpha$ by $\tilde{\varphi}$. Let $\tilde{E}$ and $\tilde{F}_i$ be inverse images of $E$ and $F_i$ on $\tilde{Z}$ respectively, and we identify them with proper transforms of $\tilde{E}$ and $\tilde{F}_i$ on $Z$ respectively. Write $\{G_j\}$ for exceptional divisors of $h$. Define $M$ to be a sum of $\tilde{E}$, $\tilde{F}_i$'s and $G_j$'s so that $M$ is the boundary divisor of $Z$. Consider a divisor $\frac{2}{n}\varphi^*L + K_Z + M$, and write it as a linear combination of components of $M$:
$$
\frac{2}{n}\varphi^*L + K_Z + M = a\tilde{E} + \sum_{i=1}^n b_i \tilde{F}_i + \sum_j c_jG_j.
$$
\begin{lemm}
\label{lemm:coefficients}
We have 
$$
a = 0 , \quad b_i > 0, \quad \text{and} \quad c_j>0.
$$
\end{lemm}

\begin{proof}
It follows from \cite[Corollary 2.31]{KM98} that a pair $(Y_\alpha, E + \sum_i F_i)$ is log canonical. 
Then \cite[Proposition 20.2 and Proposition 20.3]{Ko92} implies that a pair $(\tilde{Z}, \tilde{E} + \sum_i \tilde{F}_i)$ is also log canonical. We obtain that
$$
\frac{2}{n}\varphi^*L + K_Z + M = \frac{2}{n}\varphi^*L + h^*(K_{\tilde{Z}} + \tilde{E} + \sum_i \tilde{F}_i) + \sum_j d_j G_j,
$$
where $d_j \geq 0$. Then again it follows from \cite[Proposition 20.2]{Ko92} that
$$
\frac{2}{n}\varphi^*L + K_Z + M = \varphi^*\left(\frac{2}{n}L + K_{Y_\alpha} + E + \sum_i F_i\right) + \sum_j d_j G_j.
$$
Now Proposition~\ref{prop:coefficients} implies that $a = 0$ and $b_i > 0$. Next we prove that $c_j > 0$. If $h(G_j) \subset \cup_i \tilde{F}_i$, then the support of $\varphi^*\left(\frac{2}{n}L + K_{Y_\alpha} + E + \sum_i F_i\right)$ contains $G_j$ so $c_j > 0$ follows. Suppose that $h(G_j) \not\subset \cup_i \tilde{F}_i$. Let $V = Y_\alpha \setminus \cup_i \tilde{F}_i$. It follows from \cite[Corollary 2.31]{KM98} that $(V, E)$ is purely log terminal. Hence \cite[Proposition 20.2 and Proposition 20.3]{Ko92} show that a pair $(\tilde{\varphi}^{-1}(V), \tilde{E})$ is also purely log terminal. Thus we obtain that $d_j >0$.
\end{proof}

\section{Asymptotic volume of height balls}
\label{sec:height balls}

In this section, we recall the results of \cite{volume} and apply them to study the volume function $V(T)$ defined in Section~\ref{sec:methodofmixing}.

\subsection{Clemens complexes and height zeta functions}
\label{subsec:clemens}
Let $F$ be a local field of characteristic zero and we fix an algebraic closure $F \subset \bar{F}$. Consider a smooth projective variety $X$ defined over $F$ with a reduced effective divisor $D$ over $F$. Let 
$$
D_{\bar{F}} = \cup_{\alpha \in \bar{\mathcal A}} D_\alpha,
$$
be the irreducible decomposition of $D_{\bar{F}}$. We assume that $D_{\bar{F}} = \sum_{\alpha \in \bar{\mathcal A}} D_\alpha$ is a divisor with strict normal crossings. For any $A \subset \bar{\mathcal A}$, define $D_A = \cap_{\alpha \in A} D_\alpha$. Note that $D_\emptyset = X_{\bar{F}}$. Because of strict normal crossings, $D_A$ is a disjoint union of smooth projective varieties of codimension $|A|$. We define the geometric Clemens complex $\mathcal C_{\bar{F}}(D)$ as the set of all pairs $(A, Z)$ where $A \subset \bar{\mathcal A}$ and $Z$ is an irreducible component of $D_A$(\cite[Section 3.1.3]{volume}). The geometric Clemens complex is endowed with the following partial order relation: $(A, Z) \prec (A', Z')$ if $A \subset A'$ and $Z \supset Z'$. Thus, $\mathcal C_{\bar{F}}(D)$ is a poset and its elements are called faces. When $a \prec b$, we say that $a$ is a face of $b$. The Galois group $\mathrm{Gal}(F)$ acts on $\mathcal C_{\bar{F}}(D)$ naturally. Since $F$ is a pe
 rfect field, an integral scheme $Z$ of $X_{\bar{F}}$ is defined over $F$ if and only if $Z$ is fixed by the Galois action $\rm{Gal}(F)$. We define the rational Clemens complex $\mathcal C_F(D)$ as the sub-poset of $\mathcal C_{\bar{F}}(D)$ consisting of $\mathrm{Gal}(F)$-fixed faces(\cite[Section 3.1.4]{volume}). For any $(A, Z) \in \mathcal C_F(D)$, $D_A$ and $Z$ are defined over $F$, so $Z(F)$ makes sense. We define the analytic Clemens complex $\mathcal C_F^{\mathrm{an}}(D)$ to be a sub-poset of $\mathcal C_F(D)$ consisting of pairs $(A, Z) \in \mathcal C_F(D)$ such that $Z(F) \neq \emptyset$.(\cite[Section 3.1.5]{volume}). In general, let $\mathcal P$ be a poset. The dimension of a face $p \in \mathcal P$ is defined as the supremum of the lengths $n$ of chains $p_0 \prec \cdots \prec p_n$, where $p_i$ are distinct and $p = p_n$. The dimension of $\mathcal P$ is the supremum of all dimensions of all faces.

Let $\mathcal A = \bar{\mathcal A} / \rm{Gal}(F)$ be the quotient of $\bar{\mathcal A}$ by the Galois group $\rm{Gal}(F)$. This set can be identified with the set of irreducible components of $D$. For each $\alpha \in \mathcal A$, we denote the corresponding divisor by $\Delta_\alpha$. Let $\mathcal L = (L, \| \cdot \|)$ be a metrized line bundle with a global section $\mathsf f_L$ whose support coincide with the support of $D$ (see Section~\ref{sec:heights} for a definition of metrized line bundles). Let $\omega$ be a non-vanishing top degree differential form on $U = X\setminus D$. We are interested in the following height zeta function:
$$
\mathsf Z(s) = \int_{U(F)} \| \mathsf f_L \|^s \, \mathrm d |\omega|,
$$
where $s$ is a complex number and $|\omega|$ is a measure associated to $\omega$(see \cite[Section 2.1.7]{volume} for a definition). The connection between height zeta functions and asymptotic volume of height balls is given by Tauberian theorems \cite[Appendix A]{volume}. 
\begin{theo}{\cite[Theorem A.1]{volume}}
\label{theo:tauberian}
Suppose that $\mathsf Z(s)$ admits a meromorphic continuation to the half plane $\{ \Re(s) > a - \delta \}$, where $a >0$ and $\delta >0$, with the unique pole at $s = a$ of order $b$. Then the volume function
\[
V(T) = \int_{\{ \sH_{\mathsf f_L} (P) = \| \mathsf f_L(P) \|^{-1} \leq T \}} \, \mathrm d |\omega|
\]
behaves like $\Theta T^a \log(T)^{b-1}$ as $T \rightarrow \infty$.
\end{theo}

Hence the meromorphic continuation of $\mathsf Z(s)$ is  the key to understand the asymptotic behavior of volume of height balls, and its properties are governed by Clemens complexes of $D$. More precisely write
$$
\rm{div}\,(\mathsf f_L)= \sum_{\alpha \in \mathcal A} \lambda_\alpha \Delta_\alpha, \quad -\rm{div}\,(\omega) = \sum_{\alpha \in \mathcal A} \kappa_\alpha \Delta_\alpha.
$$
Note that we are assuming that $\lambda_\alpha > 0$ for any $\alpha \in \mathcal A$. Define
$$
a(L, \omega) = \max_{\alpha \in \mathcal A} \frac{\kappa_\alpha - 1}{\lambda_\alpha}.
$$
Then $\mathsf Z(s)$ is holomorphic when $\Re (s) > a(L, \omega)$(\cite[Lemma 4.1]{volume}). Let $\mathcal A(L, \omega)$ denote the set of all $\alpha \in \mathcal A$ where the maximum is obtained, i.e., $a(L, \omega) = (\kappa_\alpha-1)/\lambda_\alpha$. Let $\mathcal C_{F, (L, \omega)}^{\rm{an}}(D)$ be a subposet of $\mathcal C_F^{\rm{an}}(D)$ consisting of $(A, Z)$ such that $A \subset \mathcal A(L, \omega)$. \cite[Proposition 4.3 and Corollary 4.4]{volume} claims that the height zeta function $\mathsf Z(s)$ admits a meromorphic continuation extended to a half plane $\Re(s) > a(L, \omega) - \delta$ for some $\delta > 0$ and its order of the pole at $s = a(L, \omega)$ is given by 1+ the dimension of the poset $\mathcal C_{F, (L, \omega)}^{\rm{an}}(D)$. We summarize the above discussion in the following theorem:
\begin{theo}{\cite[Lemma 4.1, Proposition 4.3, and Corollary 4.4]{volume}}
\label{theo:heightzeta}
The height zeta function $\mathsf Z(s)$ is holomorphic on a half plane $\Re (s) > a(L, \omega)$. Moreover, it admits a meromorphic continuation extended to a half plane $\Re(s) > a(L, \omega) - \delta$ for some $\delta > 0$ and the order of the pole at $s = a(L, \omega)$ is
$$
1 + \dim \mathcal C_{F, (L, \omega)}^{\rm{an}}(D).
$$
\end{theo}

\subsection{Asymptotic volume}
\label{subsec:volume}
We retain the notations in Section~\ref{sec:heights} and Section~\ref{sec:moduli}. For our arithmetic applications, we need to construct moduli spaces $\Moduli_{0,n}(\bP^1,1)$ and $\Moduli_{0,n}$ over $\rm{Spec}(\mathbb Q)$. This is done in \cite{Ba08}. In fact, the moduli spaces of stable maps are constructed over $\mathrm{Spec}(\bZ)$ via geometric invariant theory.

Let $[f] \in \bP(W_n)(\bQ)$ be a binary form of degree $n$ with $\bQ$-coefficients and distinct roots. Then $X_f$ is the Zariski closure of the $\SL_2$-orbit of $[f]$ and it is defined over $\bQ$. We consider the $\SL_2$-equivariant birational morphism
$$
\varrho : \tilde{X} = \Moduli_{0,n}(\bP^1,1)/\mathfrak S_n \ra \bP(W_n),
$$
which is a $\mathfrak S_n$-quotient of the evaluation map $\rm{ev} : \Moduli_{0,n}(\bP^1,1) \ra (\bP^1)^n$. We denote the quotient map $\Moduli_{0,n}(\bP^1) \ra \tilde{X}$ by $q$. Let $\tilde{X}_f \subset \tilde{X}$ be the strict transform of $X_f$ which is again defined over $\bQ$. Write $F_f$ for the splitting field of $f$ and $\alpha = (\alpha_1, \cdots, \alpha_n)$ for roots of $f$. Then the pointed rational curve $C_\alpha = (\bP^1, \alpha_1, \cdots, \alpha_n)$ is defined over $F_f$, hence $Y_\alpha = \psi^{-1}(C_\alpha)$ is defined over $F_f$ where $\psi : \Moduli_{0,n}(\bP^1,1) \ra \Moduli_{0,n}$ is the forgetting map. Define
$$
A[n] = \tilde{X}_f \cap \tilde{D}[n], \quad A[n-1] = \tilde{X}_f \cap \tilde{D}[n-1].
$$
Note that since $\tilde{D}[n]$ and $\tilde{D}[n-1]$ are defined over $\bQ$, $A[n]$ and $A[n-1]$ are also defined over $\bQ$. The divisor $A[n]$ is geometrically irreducible, but $A[n-1]$ may not be. We have
\begin{lemm}
\label{lemm:dense}
The set $A[n](\bQ)$ is Zariski dense in $A[n]$.
\end{lemm}
\begin{proof}
Let $(C, \alpha_1, \cdots, \alpha_n, \mu)$ be a stable map such that
\begin{itemize}
\item
$C$ is the union of two $\bP^1$s and both $\bP^1$s are defined over $\bQ$,
\item
marked points $\alpha_1, \cdots, \alpha_n$ are on the collapsed component,
\item
a map $\mu$ is also defined over $\bQ$.
\end{itemize}
Then $(C, \alpha, \mu)$ is a stable map defined over the splitting field $F_f$ and it corresponds to a $F_f$-rational point $P$ on $\Moduli_{0,n}(\bP^1,1)$. Consider a Galois action $\sigma \in \mathrm{Gal}(F_{f}/\bQ)$. A point $\sigma P$ corresponds to $(C, \sigma \alpha, \mu)$, hence we have
$q(P) = q(\sigma P)$. This means that $q(P)$ is $\mathrm{Gal}(F_f/\bQ)$-fixed, thus $q(P)$ is a $\bQ$-rational point. Now we can vary the intersection of $C$ and the value of $\mu$ so that $A[n](\bQ)$ is Zariski dense.
\end{proof}

Let $Z$ be a smooth $\SL_2$-equivariant compactification of $\SL_2$ defined over $\bQ$. We have a $\SL_2$-equivariant rational map $\varphi : Z \dashrightarrow \tilde{X}_f$ mapping $\SL_2 \ni g \mapsto g[f] \in X_f$ and after applying $\SL_2$-equivariant resolution, if necessary, we may assume that $\varphi$ is an honest morphism. We denote the morphism from $Z \ra X_f$ by $\varphi$ too. Write $\omega$ for the top invariant differential form on $\SL_2$. Let $S$ be a finite set of places including the archimedean place. For each $v \in S$, we are interested in the following height zeta function: 
$$
\mathsf Z_v(s) = \int_{Z(\bQ_v)} \mathsf H_v(\varphi(z))^{-s} \, \mathrm d |\omega|_v(z)
$$
where $\mathsf H_v$ is the local height function defined in Section~\ref{sec:heights}. We have
\begin{theo}
\label{theo:analyticproperties}
Assume that either
\begin{itemize}
\item
$n \geq 5$ and $f$ is general enough so that $Y_\alpha \cong \tilde{X}_f \otimes F_f$, or
\item
all roots of $f$ are $\bQ$-rational.
\end{itemize}
Then the height zeta function $\mathsf Z_v(s)$ is holomorphic on a half plane $\Re(s) > \frac{2}{n}$ and it admits a meromorphic continuation extended to a half plane $\Re(s) > \frac{2}{n} - \delta$ for some $\delta >0$. Moreover the order of the pole at $s = \frac{2}{n}$ is 1.
\end{theo}
\begin{proof}
Suppose that $n \geq 5$ and $f$ is general. 
Let $\{ \Delta_\alpha \}_{\alpha \in \mathcal A}$ be the irreducible decomposition of the boundary divisor $D$ of $Z$. Let $\mathsf f$ be the pullback of the discriminant divisor on $Z$. Let 
\[
\mathrm{div} (\mathsf f) = \sum_{\alpha \in \mathcal A} \lambda_\alpha \Delta_\alpha, \quad -\mathrm{div}(\omega) = \sum_{\alpha \in \mathcal A} \kappa_\alpha \Delta_\alpha
\]
Then Lemma~\ref{lemm:coefficients} implies that 
\[
\max_{\alpha \in \mathcal A} \left\{\frac{\kappa_\alpha - 1}{\lambda_\alpha} \right\} = \frac{2}{n}.
\]
Moreover, the Clemens complex $C^{\mathrm{an}}_{\bQ_v, (L, \omega)}(D)$ consists of one element corresponding to $A[n]$; indeed, Lemma~\ref{lemm:dense} implies that $A[n]$ has a dense set of rational points, and Lemma~\ref{lemm:ramified} concludes that the ramification divisor $\tilde{E}$ above $A[n]$ also has a dense set of rational points. This means that this ramification divisor $\tilde{E}$ is an element of the analytic Clemens complex. Lemma~\ref{lemm:coefficients} guarantees that $\tilde{E}$ is the only divisor which achieves $\frac{\kappa_\alpha - 1}{\lambda_\alpha} = \frac{2}{n}$. Now our assertion follows from Theorem~\ref{theo:heightzeta}.

Assume that all roots of $f$ are $\bQ$-rational. Then the $Y_\alpha$ is defined over $\bQ$ and the divisor $E = Y_\alpha \cap B_N$ contains Zariski dense $\bQ$-rational points. Thus our assertion follows from Lemma~\ref{lemm:ramified}, Lemma~\ref{lemm:coefficients} and Theorem~\ref{theo:heightzeta}.
\end{proof}

We define
$$
G = \prod_{v \in S} \SL_2(\bQ_v)
$$
and consider the height ball
$$
\mathsf B(T) = \{ g \in G \mid \mathsf H(g \cdot [f]) \leq T\}.
$$
where $\mathsf H$ is the global height function defined by $\mathsf H = \prod_{v \in S} \mathsf H_v$. Let $\mu_G = \prod_{v \in S} |\omega|_v$ be a Haar measure. We denote the volume function of the height ball $\mathsf B(T)$ with respect to a Haar measure $\mu_G$ by $V(T)$. Then we have
\begin{coro}
\label{coro:precise}
Suppose that $n \geq 5$ and $f$ is general, or all roots of $f$ are $\bQ$-rational. Then we have
$$
V(T) \sim c T^{\frac{2}{n}}(\log T)^{\# S -1},
$$
for some $c>0$.
\end{coro}
\begin{proof}
To prove this corollary, we need to consider the following height zeta function:
\[
\int_{G} \sH(g \cdot [f])^{-s} \, \mathrm d \mu_G = \prod_{v \in S} \int_{\SL_2(\bQ_v)} \sH_v(g_v \cdot [f])^{-s} \, \mathrm d |\omega|_v =: \prod_{v \in S} \sZ_v(s).
\]
Then it follows from Theorem~\ref{theo:analyticproperties} that this zeta function has a pole at $s = 2/n$ of order $\#S$. If $S$ consists of the real place, then we can apply Theorem~\ref{theo:tauberian} to conclude our assertion. However when $S$ contains a non-archimedean place $p$, then the local zeta function $\sZ_p(s)$ is $\frac{2\pi i}{\log p}$-periodic and we cannot apply Theorem~\ref{theo:tauberian} since $\sZ_p(s)$ has infinitely many poles on the vertical line $\Re(s) = 2/n$. More precisely the following function
\[
(1-p^{-(s-\frac{2}{n})})\sZ_p(s)
\]
admits an analytic continuation to the half plane $\{\Re(s) > \frac{2}{n} - \delta \}$ for some $\delta >0$(\cite[Proposition 4.2]{volume}).
Instead, we apply \cite[Theorem A.7]{volume} to $\sZ(s)$. To apply this theorem, one needs to verify that $\log p / \log q$ is not Liouville number for distinct primes $p, q$. This fact is proved in \cite{Baker}. \cite[Theorem 3.1]{Baker} claims that the irrationality measure of $\frac{\log p}{\log q}$ is bounded by a constant depending on $p, q$. Thus our assertion follows.
\end{proof}

\section{Generalizations}
\label{sect:generalizations}

The method to prove the main result in Section~\ref{sec:methodofmixing} generalizes to semisimple groups. Let $F$ be a number field and $G$ a simply connected, almost $F$-simple group. Let $S$ be a finite set of places containing all archimedean places $v$ such that $G(F_v)$ is non-compact. We denote the ring of integers of $F$ by $\mathfrak o_{F}$ and the ring of $S$-integers of $F$ by $\mathfrak o_{F,S}$. We fix an integral model of $G$ so that $G(\mathfrak o_{F})$ makes sense. Denote the $S$-adic Lie group $\prod_{v \in S} G(F_v)$ by $G_S$. We embed $G(\mathfrak o_{F, S})$ into $G_S$ diagonally. Then $G(\mathfrak o_{F, S})$ is a lattice in $G_S$. Let $X$ be a smooth projective equivariant compactification of $G$ defined over $F$ and $\mathcal L = (L, \| \cdot \|)$ an adelically metrized big line bundle on $X$ with a global section $s$ whose support coincides with $X \setminus G$. We define local height functions $\mathsf H_v : G(F_v) \ra \bR_{>0}$ and the global height $\mathsf H : G_S \ra \bR_{>0}$ by
$$
\mathsf H_{\mathcal L, s, v}(P_v) = \| s(P_v) \|^{-1}, \quad \mathsf H_{\mathcal L, s}((P_v)_{v \in S}) = \prod_{v \in S} \mathsf H_{\mathcal L, s, v}(P_v).
$$
We suppose that for any archimedean place $v \in S$, the local height function $\mathsf H_{\mathcal L, s, v}$ is invariant under the action of a maximal compact subgroup $K_v$. It is always possible to choose a metrization to satisfy this property. It is also a property of height functions that for any non-archimedean place $v$, the local height $\mathsf H_{\mathcal L, s, v}$ is invariant under the action of a compact open subgroup $K_v$. We are interested in a counting function $N(T)$ of $G(\mathfrak o_{F,S})$ with respect to $\mathsf H_{\mathcal L, s}$,
$$
N(T) = \# \{ \gamma \in G(\mathfrak o_{F, S}) \mid \mathsf H_{\mathcal L, s}(\gamma) \leq T \}.
$$
When $X$ is a biequivariant compactification of $G$, this counting function has been studied in \cite{TBT11} and \cite{BO12}. However, the case of one-sided equivariant compactifications remained open. Our technique in Section~\ref{sec:methodofmixing} can solve this case.

The action of $G_S$ on $Y := G(\mathfrak o_{F,S}) \backslash G_S$ is mixing(\cite[Proposition 2.4]{BO12}). We define
$$
H_S = \prod_{v \in S} K_v.
$$
Then $G_S$ has the wavefront property in $H_S \backslash G_S$. Thus translates of $H_S$-orbits are equidistributed in $Y$(Theorem~\ref{theo:equidistribution}). Let $\mu_S$ be a haar measure on $G_S$ and $\mu_Y$ an invariant measure on $Y$ such that $\mu_S = \mu_Y$ holds locally. We consider height balls 
$$
\mathsf B(T) = \{ g \in G_S) \mid \mathsf H_{\mathcal L, s}(g) \leq T \}.
$$
We denote the volume function of these height balls by $V(T)$. Now the discussion in Section~\ref{sec:methodofmixing} leads to the following theorem:
\begin{theo}
\label{theo:general}
Let $V^*(T) = V(T)/\mu_Y(Y)$. Then we have
$$
\frac{N(T)}{V^*(T)} \ra 1 \quad \text{as $T\ra +\infty$}.
$$
\end{theo}

\bibliographystyle{alpha}
\bibliography{binaryforms}

\end{document}